\documentclass[12pt]{article}
\usepackage[matrix,arrow,curve]{xy}
\usepackage{amssymb,amsthm}
\usepackage{supertabular}
\usepackage{lscape}
\usepackage{pstricks,pst-node,pst-plot}

\newcommand{\ddate}{\small March 15, 2010}

\newtheorem{dummy}{anything}[section]
\newtheorem{Theorem}[dummy]{Theorem}
\newtheorem{Lemma}[dummy]{Lemma}
\newtheorem{Proposition}[dummy]{Proposition}
\newtheorem{Corollary}[dummy]{Corollary}

\newtheorem{Remark}[dummy]{Remark}

\newtheorem{ccote}[dummy]{}
\newtheorem*{thmA}{Theorem A}
\newtheorem*{thmB}{Theorem B}

\newcommand{\mancqfd}{\hfill \ensuremath{\Box}}

\newcommand{\unm}{\{1,\dots,n\}}
\newcommand{\unmi}{\{1,\dots,n-1\}}
\newcommand{\crit}{{Crit}}

\newcommand{\bbr}{{\mathbb R}}
\newcommand{\bbs}{{\mathbb S}}

\newcommand{\bbc}{{\mathbb C}}
\newcommand{\bbz}{{\mathbb Z}}

\newcommand{\bbn}{{\mathbb N}}

\newcommand{\calb}{{\mathcal B}}
\newcommand{\calc}{{\mathcal C}}

\newcommand{\cali}{{\mathcal I}}
\newcommand{\calj}{{\mathcal J}}

\newcommand{\call}{{\mathcal L}}

\newcommand{\caln}{{\mathcal N}}

\newcommand{\calr}{{\mathcal R}}
\newcommand{\cals}{{\mathcal S}}

\newcommand{\algn}{\mathfrak{N}}

\newcommand{\pcirc}{\kern .7pt {\scriptstyle \circ} \kern 1pt}

\newcommand{\mun}{{-1}}

\newcommand{\nua}[2]{\caln^{#1}_{#2}}

\newcommand{\cha}[2]{\calc^{#1}_{#2}}
\newcommand{\chav}{\cha{}{}}

\newcommand{\calss}{\cals}
\newcommand{\shod}{\dot\cals}
\newcommand{\shoti}{\tilde\cals}
\newcommand{\bcals}{\calb\cals}

\newcommand{\scr}{\scriptscriptstyle}
\renewcommand{\:}{\colon}

\newcommand{\sk}[1]{\vskip #1 mm}
\newcommand{\eqref}[1]{(\ref{#1})}
\newcommand{\hfl}[2]{\smash{\mathop{\hbox to 1 truecm{\kern %
3pt\rightarrowfill\kern 3pt}}%
\limits^{\scriptstyle#1}_{\scriptstyle#2}}}
\newcommand{\cqfd}{\unskip\kern 6pt\penalty 500%
\raise -2pt\hbox{\vrule\vbox to10pt{\hrule width %
4pt\vfill\hrule}\vrule}\smallskip}
\newcommand{\proref}[1]{Proposition~\ref{#1}}
\newcommand{\remref}[1]{Remark~\ref{#1}}
\newcommand{\lemref}[1]{Lemma~\ref{#1}}
\newcommand{\corref}[1]{Corollary~\ref{#1}}
\newcommand{\thref}[1]{Theorem~\ref{#1}}

\newcommand{\secref}[1]{Section~\ref{#1}}

\newcommand{\llangle}[2]{\langle #1 ,#2\rangle}

\newsavebox{\boxJ}\newsavebox{\boxbarJ}

\title{
On the cohomology ring of chains in $\bbr^d$ \\
{\small (Walker's conjecture for chains in $\bbr^d$)}}
\author{Michael FARBER, Jean-Claude HAUSMANN and Dirk SCH\"UTZ}
\date{\ddate}

\begin{document}
\maketitle 

\begin{abstract}
A chain is a configuration in $\bbr^d$ of segments of length $\ell_1,\dots\ell_{n-1}$
joining two points at a fixed distance $\ell_n$. 
When $d\geq 3$, we prove that the spaces of such chains are determined up to diffeomorphism 
by their (${\rm mod\,}2$)-cohomology rings.
\end{abstract}

\begin{center}
{\small MSC classes: 	55R80, 70G40, 57R19}
\end{center}

\section*{Introduction}

For $\ell=(\ell_1,\dots,\ell_n)\in\bbr_{>0}^n$ and $d$ a positive integer, define the subspace
$\cha{n}{d}(\ell)$ of $(S^{d-1})^{n-1}$ by
$$
\cha{n}{d}(\ell)=\big\{z=(z_1,\dots,z_{n-1})\in (S^{d-1})^{n-1}
\mid \sum_{i=1}^{n-1}\ell_iz_i=\ell_n\, e_1
\big\} \, ,
$$
where $e_1=(1,0,\dots,0)$ is the first vector of the standard
basis $e_1,\dots,e_d$ of $\bbr^d$.
An element of $\cha{n}{d}(\ell)$, called a {\it chain}, can be visualised 
as a configuration of $(n-1)$-segments in $\bbr^d$, of length
$\ell_1,\dots,\ell_{n-1}$, joining the origin to $\ell_n e_1$.
The vector $\ell$ is called the {\it length-vector}.

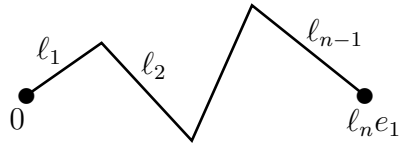
\sk{15}
\hskip 50mm
\begin{minipage}{6cm}
\setlength{\unitlength}{.05mm}
\begin{pspicture}(0,-2.1)(0,0)
\put(-2,0){\circle*{0.2}} \uput[-95](-2.1,0){$0$}
\put(2.5,0){\circle*{0.2}} \uput[-88](2.6,0){$\ell_ne_1$}

\uput[135](-1.4,0.3){$\ell_1$}
\uput[45](-0.6,0.1){$\ell_2$}
\uput[45](1.6,0.5){$\ell_{n-1}$}

\psline[linewidth=1pt](-2,0)(-1,0.7)(0.2,-0.6)(1,1.2)(2.5,0)
\end{pspicture}
\end{minipage}

The group $O(d-1)$,
viewed as the subgroup of $O(d)$ stabilising the first axis, acts naturally
(on the left) upon $\cha{n}{d}(\ell)$.
The quotient $SO(d-1)\big\backslash\cha{n}{d}(\ell)$ is the
{\it polygon space} $\nua{n}{d}$, usually defined as
$$
\nua{n}{d}(\ell)=
SO(d)\Bigg\backslash\Big\{z\in (S^{d-1})^n\,\Bigg|\,\sum_{i=1}^n\,\ell_i z_i=0\Big\} \, .
$$
When $d=2$ the space of chains $\cha{n}{2}(\ell)$ coincides with the
polygon space $\nua{n}{2}(\ell)$. Descriptions of several chain and polygon 
spaces are provided in \cite{Ha2}
(see also \cite{Ha} for a classification of $\cha{4}{d}(\ell)$).

A length-vector $\ell\in\bbr_{>0}^n$ is {\it generic}
if $\cha{n}{1}(\ell)=\emptyset$, that is to say there is no collinear chain.
It is proven in e.g.~\cite{Ha} that, for $\ell$ generic, 
$\cha{n}{d}(\ell)$ is a smooth closed
manifold of dimension 
$$
\dim \cha{n}{d}(\ell) = (n-2)(d-1)-1 \, .
$$ 
Another known fact fact is that if $\ell,\ell'\in\bbr_{>0}^n$ satisfy 
$$
(\ell_1',\dots,\l_{n-1}',\ell_n')= (\ell_{\sigma(1)},\dots,\l_{\sigma(n-1)},\ell_n)
$$
for some permutation $\sigma$ of $\unmi$, then $\cha{n}{d}(\ell')$ and $\cha{n}{d}(\ell)$
are $O(d-1)$-equivariantly diffeomorphic, see \cite[1.5]{Ha2}.

A length-vector $\ell\in\bbr_{>0}^n$ is {\it dominated} if $\ell_n\geq \ell_i$ for
all $i<n$.
The goal of this paper is to show that the spaces $\cha{n}{d}(\ell)$ are, for $\ell$ generic
and dominated and when $d\geq 3$, determined up to 
$O(d-1)$-equivariant diffeomorphism by their ${\rm mod}2$-cohomology rings:

\begin{thmA}
Let $d\in\bbn$, $d\geq 3$. 
 Then, the following properties of generic and dominated length-vectors $\ell, \ell'\in\bbr_{>0}^n$ are equivalent:
\renewcommand{\labelenumi}{(\alph{enumi})}
\begin{enumerate}
\item $\cha{n}{d}(\ell)$ and $\cha{n}{d}(\ell')$ are $O(d-1)$-equivariantly diffeomorphic.
\item $H^*(\cha{n}{d}(\ell);\bbz)$ and $H^*(\cha{n}{d}(\ell');\bbz)$ are isomorphic as graded rings.
\item $H^*(\cha{n}{d}(\ell);\bbz_2)$ and $H^*(\cha{n}{d}(\ell');\bbz_2)$ are isomorphic as graded rings.
\end{enumerate}
\end{thmA}

In the case $d=2$ we do not know if (c) $\Rightarrow$ (a) although the equivalence (a) $\sim$ (b) is true. This is related to a conjecture of K. ~Walker \cite{Wa} who suggested that
planar polygon spaces are determined by their
integral cohomology rings. 
The conjecture was proven for a large class of length-vectors
in \cite{FHS} and the (difficult) remaining cases were settled in~\cite{Sz}. 
The spatial polygon spaces $\nua{n}{3}$ are also determined up to diffeomorphism
by their ${\rm mod}2$-cohomology ring if $n>4$, see \cite[Theorem~3]{FHS}. 
No such result is known for $\nua{n}{d}$ when $d>3$.

We now give the scheme of the proof of Theorem~A. We first recall that the
$O(d-1)$-diffeomorphism type of $\cha{n}{d}(\ell)$ is determined by $d$ and the sets of
$\ell$-short (or long) subsets, which play an important role all along this paper.
A subset $J$ of $\unm$ is  {\it $\ell$-short}, or just {\it short}, if
$$
\sum_{i\in J}\ell_j < \sum_{i\notin J}\ell_j \, .
$$
The reverse inequality defines {\it long} (or {\it $\ell$-long}) subsets.
Observe that $\ell$ is generic if and only if any subset of $\unm$ is either short or long.

The family of subsets of $\unm$ which are long is denoted by $\call=\call(\ell)$.
Short subsets form a poset under inclusion, which we denote by 
$\calss=\calss(\ell)$. We are interested in the subposet
\begin{equation}\label{defshod}
\shod=\shod(\ell)=\{J\subset\unmi\mid J\cup \{n\}\in\calss\}\, .
\end{equation}


The following lemma is proven in \cite[Lemma~1.2]{Ha2} (this reference uses 
the poset $\cals_n(\ell)=\{J\in\calss\mid n\in J\}$ which is determined by $\shod(\ell)$).
\begin{Lemma}\label{shortdetcha}
Let $\ell,\ell'\in\bbr_{>0}^n$ be generic length-vectors.
Suppose that $\shod(\ell)$ and $\shod(\ell')$ are isomorphic
as simplicial complexes.
Then $\cha{m}{d}(\ell)$ and $\cha{m}{d}(\ell')$ are $O(d-1)$-equivariantly
diffeomorphic. \mancqfd
\end{Lemma}


Note that $H^*(\cha{n}{d};\bbz_2)=0$ if and only if $\cha{n}{d}=\emptyset$,
which happens if and only if $\{n\}$ is long. We can thus suppose that
$\{n\}$ is short and hence $\shod(\ell)$ is determined by its
subposet
\begin{equation}\label{defshoti}
\shoti=\shoti(\ell) =\shod(\ell) - \{\emptyset\} \, .
\end{equation}
The poset $\shoti$ is an abstract simplicial complex (as a subset of a short subset is short)
with vertex set contained in $\unmi$. 
To prove Theorem A, it then suffices to show that the graded ring $H^*(\cha{n}{d}(\ell);\bbz_2)$ determines $\shoti(\ell)$ when $\ell$ is dominated. 

For a finite simplicial
complex $\Delta$ whose vertex set $V(\Delta)$ is contained in $\unm$, 
consider the graded ring
$$
\Lambda(\Delta)= \bbz_2[X_1,\dots,X_n]\big/\cali(\Delta) \, ,
$$
where $\cali(\Delta)$ is the ideal of the polynomial ring $\bbz_2[X_1,\dots,X_n]$
generated by $X_i^2$ and the monomials $X_{j_1}\cdots X_{j_k}$ when
$\{j_1,\dots,j_k\}$ is not a simplex of $\Delta$. 
Let $\Delta$ and $\Delta'$ be two finite simplicial complexes with vertex sets contained in $\unm$. By a result of J.~Gubeladze, any graded ring isomorphism 
$\Lambda(\Delta)\approx\Lambda(\Delta')$ is induced by a simplicial isomorphism
$\Delta\approx\Delta'$ (see \cite[Example~3.6]{Gu}; for more details, see \cite[Theorem~8]{FHS}). 
Hence, the implication (c) $\Rightarrow$ (a) of Theorem~A will be established if we prove
the following result:

\begin{thmB}
Let $\ell\in\bbr_{>0}^n$ be a generic length-vector which is dominated.
When $d\geq 3$, the subring $H^{(d-1)*}(\cha{n}{d}(\ell);\bbz_2)$
of $H^*(\cha{n}{d}(\ell);\bbz_2)$ is isomorphic to $\Lambda(\shoti(\ell))$.
\end{thmB}

The implication (b) $\Rightarrow$ (c) follows since, under the condition that $\ell$ is dominated, $H^{*}(\cha{n}{d}(\ell);\bbz)$ is torsion free (see Theorem \ref{T.cohcha}).
Note also Remark \ref{Rcohchawrong} which shows that the condition that $\ell$ is dominated is necessary. 

The proof of Theorem~B is given in \secref{SPthB}.
The preceding sections are preliminaries for this goal. For instance, the computation
of $H^*(\cha{n}{d}(\ell);\bbz)$ as a graded abelian group, is given
in \thref{T.cohcha}.

\section{Robot arms in $\bbr^d$}\label{S.robotarm}

Let
$$
\bbs=\bbs^n_d=\{\rho\:\unm\to S^{d-1}\} \approx (S^{d-1})^n \, .
$$
By post-composition, the orthogonal group $O(d)$ acts smoothly
on the left upon $\bbs$. 
In \cite[\S\,4--5]{FS}, the quotient $W=SO(2)\backslash\bbs^{n}_{2}\approx(S^{1})^{n-1}$ 
is used to get cohomological informations about  $\cha{n}{2}$. In this section, we
extend these results for $d>2$. The quotient $SO(d)\backslash\bbs^{n}_{d}$ is no
longer a convenient object to work with, so we replace it by the fundamental domain for the 
$O(d)$-action given by the submanifold
%
%
$$
Z=Z^n_d=\{\rho\in\bbs\mid \rho(n)=-e_1\} \approx (S^{d-1})^{n-1} \, .
$$
Observe that $Z$ inherits an action of $O(d-1)$.

For a length-vector $\ell=(\ell_1,\dots,\ell_n)\in\bbr_{>0}^n$
the {\it $\ell$-robot arm} is the smooth map $\tilde F_\ell\:\bbs\to\bbr^d$
defined by
$$
\tilde F_\ell(\rho) = \sum_{i=1}^n \ell_i\rho(i) \, .
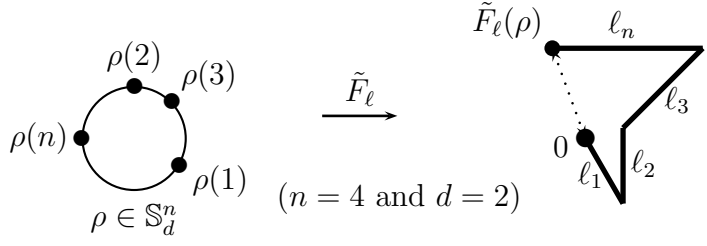
$$

\sk{20}
\hskip 50mm
\begin{minipage}{6cm}
\setlength{\unitlength}{.05mm}
\begin{pspicture}(0,-2.1)(0,0)
\pscircle(-3,0){0.7}
\put(-3.7,0){\circle*{0.2}}
\put(-3,0.7){\circle*{0.2}}
\put(-2.5,0.5){\circle*{0.2}}
\put(-2.4,-0.35){\circle*{0.2}}

\uput[-90](-3,-0.7){$\rho\in\bbs^{n}_{d}$}
\uput[180](-3.7,0){$\rho(n)$}
\uput[90](-3,0.7){$\rho(2)$}
\uput[-30](-2.4,-0.35){$\rho(1)$}
\uput[45](-2.5,0.5){$\rho(3)$}
\psline[linewidth=1pt]{->}(-0.5,0.3)(0.5,0.3)
\put(-0.2,0.5){$\tilde F_\ell$}
\put(-1.1,-0.9){($n=4$ and $d=2$)}
\put(3,0){\circle*{0.25}}
\uput[-160](2.95,0){$0$}
\psline[linewidth=2pt](3,0)(3.5,-0.866)
\put(2.9,-0.6){$\ell_1$}
\psline[linewidth=2pt](3.5,-0.866)(3.5,0.14)
\put(3.6,-0.45){$\ell_2$}
\psline[linewidth=2pt](3.5,0.14)(4.56,1.2)
\put(4.05,0.4){$\ell_3$}
\psline[linewidth=2pt]{-*}(4.56,1.2)(2.56,1.2)
\put(3.3,1.4){$\ell_n$}
\psline[linewidth=1pt,linestyle=dotted]{<->}(2.56,1.2)(3,0)
\uput[135](2.56,1.2){$\tilde F_\ell(\rho)$}
\end{pspicture}
\end{minipage}

\noindent Observe that the point $\rho\in\cha{4}{d}$ in the above figure
lies in $Z$.
We also define an $O(d)$-invariant smooth map $\tilde f_\ell\:\bbs\to\bbr$ by
$$
\tilde f_\ell(\rho) = -|F_\ell(\rho)|^2 \, .
$$
The restrictions of $\tilde F$ and $\tilde f$ to $Z$ are denoted by $F$ and $f$ respectively.
Observe that
$$
\chav=\cha{n}{d}(\ell)= f^\mun(0)\subset Z \, .
$$
Define
$$
\bbs'=\bbs -\chav \ \hbox{ and } \ Z'=Z-\chav \
$$
The restriction of
$\tilde f$ and $f$ to $\bbs'$ and $Z'$ are denoted by $\tilde f'$ and
$f'$ respectively. 

Denote by $\crit(g)$ be the subspace
of critical points of a real value map $g$.
One has $\crit(\tilde f)=\chav\,\dot\cup\,\crit(\tilde f')$
and $\crit(f)=\chav\,\dot\cup\,\crit(f')$,
where $\dot\cup$ denotes the disjoint union.
It is easy and well known that $\rho\in \crit(\tilde f')$
if and only if $\rho$ is a collinear configuration,
i.e. $\rho(i)=\pm \rho(j)$ for all $i,j\in\unm$.

We will index the critical points of $\tilde f'$ and $f'$ by the long subsets.
For each $J\in\call$, let $\crit_J(\tilde f')\subset \crit(\tilde f')$
be defined by
$$
\crit_J(\tilde f')= \big\{\rho\in\bbs\mid
\kappa_J(j)\rho(j)=\kappa_J(i)\rho(i) \ \hbox{ for all }
 i,j\in\unm
\big\} \, .
$$
where 
$\kappa_{\scr J}\:\unm\to\{\pm 1\}$ the multiplicative
characteristic function of $J$, defined by:
$$
\kappa_{\scr J}(i)=
\left\{\begin{array}{lll}
-1 & \hbox{if } i\in J \\
1 & \hbox{if } i\notin J \, .
\end{array}\right.
$$
In particular, $\kappa_{\scr \bar J}=-\kappa_{\scr J}$ if 
$\bar J$ is the complement of $J$ in $\unm$.
In words, $\crit_J(\tilde f')$ is the space of collinear
configurations $\rho$ which take constant values on $J$ and $\bar J$ and such that $\rho(J)=-\rho(\bar J)$.
The space $\crit_J(\tilde f')$
is a submanifold of $\bbs$ diffeomorphic, via $\tilde F$,
to the sphere in $\bbr^d$ of radius
$\sum_{j\in J}\ell_j - \sum_{j\notin J}\ell_j$
(this number is positive since $J$ is long).
One has
$$
\crit (\tilde f')=\dot{\bigcup_{J\in\call}}
\crit_J(\tilde f') \, .
$$
The $O(d)$-invariance of $\tilde f'$ has two consequences:
each sphere $\crit_J(\tilde f')$ intersects $Z$ transversally 
in the single point $\rho_{\scr J}$ and $\crit(f')=\crit(\tilde f') \cap Z$.
Hence
\begin{equation}\label{eqcrif}
\crit (f')=\{\rho_{\scr J}\mid J\in\call\} 
\end{equation}
(note that $\rho_{\scr J}\notin\chav$ as $\ell$ is generic).
As $\rho(n)=-e_1$ if $\rho\in Z$, the critical points $\rho_{\scr J}$
are of two types, depending on $n\in J$ or not:
\begin{equation}\label{rhoj}
\rho_{\scr J}(i)=
\left\{\begin{array}{rll}
\kappa_J(i)\, e_1 & \hbox{if } n\in J \\
-\kappa_J(i)\, e_1 &\hbox{if }  n\notin J \, .
\end{array}\right.
\end{equation}

\begin{Lemma}\label{fmorse}
The map $f'\:Z'\to (-\infty,0)$ is a proper Morse function with
set of critical points $\{\rho_{\scr J}\mid J\in\call\}$.
The index of $\rho_{\scr J}$ is $(d-1)(n-|J|)$.
\end{Lemma}

\begin{proof}
Because $f'$ extends to $f\:(Z,\chav)\to ((-\infty,0],0)$,
the map $f'$ is proper.
We already described $\crit(f')$ in~\eqref{eqcrif}.
The non-degeneracy of $\rho_{\scr J}$
and the computation of its index
are classical in topological robotics using arguments
as in \cite[proof of Theorem~3.2]{Ha}.
\end{proof}

Consider the axial involution $\hat\tau$ on $\bbr^d=\bbr\times\bbr^{d-1}$
defined by $\hat\tau(x,y)=(x,-y)$.
It induces an involution $\tau$ on $\bbs$ and on $Z$.
The maps $\tilde f$ and $f$ are $\tau$-invariant.
Moreover, the critical set of $f'\:Z'\to (-\infty,0)$
coincides with the fixed point set $Z^\tau$.
By \lemref{fmorse} and \cite[Theorem~4]{FS}, this proves the following

\begin{Lemma}\label{fperfect}
The Morse function $f'\:Z'\to (-\infty,0)$ is $\bbz$-perfect, in the sense that
$H_i(Z')$ is free abelian of rank equal to the number of
critical points of index~$i$. Moreover, the induced map
$\tau_*\:H_i(Z')\to H_i(Z')$ is multiplication by $(-1)^i$.
\mancqfd
\end{Lemma}

(Theorem~4 of \cite{FS} is stated for a Morse function $f\:M\to\bbr$
where $M$ is a compact manifold with boundary. To use it in the proof
of \lemref{fperfect}, just replace $Z'$ by $Z-N$ where $N$ is a small
open tubular neighbourhood of $\chav$.)

We now represent a homology basis for $Z$ and $Z'$ by convenient
closed manifolds. For $J\subset\unm$, define
$$
\bbs_J = \{\rho\in\bbs\mid |\rho(J)|=1\} 
$$
(the condition $|\rho(J)|=1$ is another way to say that $\rho$
is constant on $J$).
The space $\bbs_J$ is a closed submanifold of $\bbs$ diffeomorphic
to $(S^{d-1})^{n-|J|+1}$. As it is $O(d)$-invariant, it
intersects $Z$ transversally. Let
$$
W_J=\bbs_J \cap Z \approx (S^{d-1})^{n-|J|} \, .
$$
The manifold $W_J$ is $O(d-1)$-invariant and then is $\tau$-invariant.
As in Formula~\eqref{rhoj}, the dichotomy ``$n\in J$ or not'' occurs:
\begin{equation}\label{zJ}
W_{J}=
\left\{\begin{array}{lll}
\{\rho\in Z\mid \rho(J)=-e_1\} & \hbox{if } n\in J \\
\{\rho\in Z\mid |\rho(J)|=1\}
&\hbox{if }  n\notin J \, .
\end{array}\right.
\end{equation}
We denote by $[W_J]\in H_{(d-1)(n-|J|)}(Z;\bbz)$ the class represented by $W_J$
(for some chosen orientation of $W_J$).
If $J$ is long, then $W_J\subset Z'$ and we also denote by
$[W_J]$ the class represented by $W_J$ in $H_{(d-1)(n-|J|)}(Z';\bbz)$.
\begin{Lemma}\label{zjbase}
\renewcommand{\labelenumi}{(\alph{enumi})}
\begin{enumerate}
\item $H_*(Z';\bbz)$ is freely generated by the classes $[W_J]$ for $J\in\call$.
\item $H_*(Z;\bbz)$ is freely generated by the classes $[W_J]$ for all
$J\in\unm$ with $n\in J$.
\end{enumerate}
\end{Lemma}

\begin{proof}
For (a), we invoke \cite[Theorem~5]{FS}. Indeed, the 
the collection of $\tau$-invariant manifolds
$\{W_J\mid J\in\call\}$ satisfies all the hypotheses of this theorem
(see also \cite[Lemma~8]{FS}).

Let $K=\{1,\dots,n-1\}$. The restriction of $\rho\in Z$ to $K$ gives
a diffeomorphism from $h\:Z\stackrel{\approx}{\to}\bbs_K\approx (S^{d-1})^{n-1}$.
By the K\"unneth formula, $H_*(\bbs_K;\bbz)$ is freely generated by
the classes $[W_I]$ for all $I\subset K$. If $n\in J$,  $h(W_J)=W_{J-\{n\}}$,
which proves (b).
\end{proof}

Let $J,J'\subset\unm$. When $|J|+|J'|=n+1$, one has
$\dim W_J + \dim W_{J'}=\dim Z=\dim Z'$ and the intersection number
$[W_J]\cdot [W_{J'}]\in\bbz$ is defined (intersection in $Z$). We shall use the following
formulae.

\begin{Lemma}\label{intnum}
$J,J'\subset\unm$ with $|J|+|J'|=n+1$.
Then
$$
[W_J]\cdot [W_{J'}] =
\left\{\begin{array}{lll}
\pm 1 & \hbox{if } |J\cap J'|=1 \\
0 & \hbox{if } |J\cap J'|>1 \hbox{ and } n\in J\cup J'\, .
\end{array}\right.
$$
\end{Lemma}

\begin{proof}
Suppose that $J\cap J'=\{q\}$. Then $|J\cup J'|= |J|+|J'|-|J\cap J'|=n$.
Then, $n\in J\cup J'$ and $W_J\cap W_{J'}$ consists
of the single point $\rho_{\scr J\cup J'}$ (satisfying $\rho_{\scr J\cup J'}(i)=-e_1$
for all $i\in\unm$). It is not hard to check that the intersection
is transversal (see \cite[proof of~(34)]{FS}), so $[W_J]\cdot [W_{J'}]=\pm 1$.

In the case $|J\cap J'|>1$, there exists $q\in J\cap J'$ with
$q\neq n$. Let $\alpha$ be a rotation of $\bbr^d$ such that $\alpha(e_1)\neq e_1$.
Let $h\:Z\to Z$ be the diffeomorphism such that 
$h(\rho)(k)=\rho(k)$ if $k\neq q$ and $h(\rho)(q)=\alpha\pcirc\rho(q)$. 
We now use that $n\in J\cup J'$, say $n\in J'$. Then, $\rho(q)=-e_1$ for $\rho\in W_{J'}$.
Hence, $h(W_J)\cap W_{J'} =\emptyset$. As $h$ is isotopic to the identity of $Z$, this
implies that $[W_J]\cdot [W_{J'}]=0$.
\end{proof}

%

\begin{Remark} \rm In \lemref{intnum}, the hypothesis 
$n\in J\cup J'$ is not necessary if $d$ is even, by the above proof,
since there exists a diffeomorphism of $S^{d-1}$ isotopic
to the identity and without fixed point. But, for example, if $n=d=3$,
one checks that $[W_J]\cdot [W_{J'}]=\pm 2$ for $J=J'=\{1,2\}$.
\end{Remark}

In the case $n\in J\cap J'$ and $|J|+|J'|=n+1$, \lemref{intnum}
takes the following form:
\begin{equation}\label{intnumZ}
[W_J]\cdot [W_{J'}] =
\left\{\begin{array}{lll}
\pm 1 & \hbox{if } J\cap J'=\{n\} \\
0 & \hbox{otherwise} \, .
\end{array}\right.
\end{equation}

Therefore, the basis $\{[W_J]\mid |J|=n-k,n\in J\}$ of $H_{k(d-1)}(Z;\bbz)$ has a 
dual basis (up to sign)
$\{[W_J]^\sharp\in H_{(n-k)(d-1)}(Z;\bbz)\mid |J|=n-k,n\in J\}$ for
the intersection form, defined by $[W_J]^\sharp=[W_K]$, where 
$K=\bar J\cup\{n\}$.

We are now in position to express the homomorphism 
$\phi_*: H_*(Z';\bbz)\to H_*(Z;\bbz)$ induced by the inclusion $Z'\subset Z$.
By \lemref{zjbase}, one has a direct sum decomposition
$$
H_*(Z';\bbz) = A_* \oplus B_* \, ,
$$
where
\begin{itemize}
\item $A_*$ is the free abelian group generated by $[W_J]$
with $J\subset\unm$ long and $n\in J$.
\item $B_*$ is the free abelian group generated by $[W_J]$
with $J\subset\unm$ long and $n\notin J$.
\end{itemize}
\lemref{zjbase} also gives a direct sum decomposition 
$$
H_*(Z;\bbz) = A_* \oplus C_* \, ,
$$
where
\begin{itemize}
\item $A_*$ is the free abelian group generated by $[W_J]$
with $J\subset\unm$ with $n\in J$ and $J$ long.
\item $C_*$ is the free abelian group generated by $[W_J]$
with $J\subset\unm$ with $n\in J$ and $J$ short.
\end{itemize}

\begin{Lemma}\label{phi}
\renewcommand{\labelenumi}{(\alph{enumi})}
\begin{enumerate}
\item $\phi_*$ restricted to $A_*$ coincides with the identity of $A_*$.
\item Suppose that $\ell$ is dominated. Then $\phi_*(B_*)\subset A_*$.
\end{enumerate}
\end{Lemma}

\begin{proof}
Point (a) is obvious. For (b),
let $[W_J]\in B_{(n-|J|)(d-1)}$.
By what has been seen, it suffices to show that 
$[W_J]\cdot [W_K]^\sharp=0$ for all $[W_K]\in  C_{*}$
with $|K|=|J|$.
Suppose that there exists $[W_K]\in C_{*}$
with $|K|=|J|$ 
such that $[W_J]\cdot [W_K]^\sharp=\pm 1$. One has
$[W_K]^\sharp=[W_L]$ where $L=\bar K\cup\{n\}$. By \lemref{intnum},
this means that $J\,\cap\, (\bar K\cup\{n\})=J-K=\{i\}$, with $i<n$.
As $|K|=|J|$,
this is equivalent to $K= (J-\{i\})\cup \{n\}$. As $\ell_n\geq\ell_j$,
this contradicts the fact that $J$ is long and $K$ is short.
\end{proof}

\section{The Betti numbers of the chain space}\label{S.cohom}

Let $\ell=(\ell_1,\dots,\ell_n)$ be a dominated length-vector.
Let $a_k=a_k(\ell)$ be the number of short subsets $J$ containing $n$
with $|J|=k+1$. Alternatively, $a_k$ is the number of sets $J\in\shod$
with $|J|=k$.

\begin{Theorem}\label{T.cohcha}
Let $\ell=(\ell_1,\dots,\ell_n)$ be a dominated length-vector.
Then, if $d\geq 3$, $H^*(\cha{n}{d}(\ell);\bbz)$ is free abelian with rank
$$
{\rm rank\,} H^k(\cha{n}{d}(\ell);\bbz) =
\left\{
\begin{array}{lll}
a_{s} & \hbox{if } \kern 2mm k=s(d-1), \quad s=0,1, \dots, n-3, \\
a_{n-s-2} & \hbox{if } \kern 2mm  k = s(d-1) -1, \quad s=0, \dots, n-2,\\
0 & \hbox{otherwise.}
\end{array}\right.
$$
\end{Theorem}

\begin{proof}
Let $N$ be a closed tubular neighbourhood of $\chav=\cha{n}{d}(\ell)$ in
$Z=Z^n_d$. Let $Z'=Z-\chav$.
By Poincar\'e-Lefschetz duality and excision, one has the isomorphisms
on integral homology
$$
H^k(\chav)\approx H^k(N) \approx H_{(n-1)(d-1)-k}(N,\partial N)
\approx  H_{(n-1)(d-1)-k}(Z,Z')
$$
and
\begin{eqnarray*}
H^k(Z,\chav) &\approx & H^k(Z,N) \approx H^k(Z- {\rm int\,} N,\partial N)
\\ 
&\approx & H_{(n-1)(d-1)-k}(Z- {\rm int\,} N)\approx
H_{(n-1)(d-1)-k}(Z') \, .
\end{eqnarray*}

The homology of $Z$ and $Z'$ are concentrated in degrees which are multiples of 
$(d-1)$. Hence, $H^k(\chav)=0$ if $k\not\equiv 0,-1 \ {\rm mod\,} (d-1)$.
The possibly non-vanishing $H^k(\chav)$ sit in a diagram
$$
\begin{array}{c}{\xymatrix@C-10pt@M-1pt@R-4pt{
0 \ar[r] &
H^{s(d-1)-1}(\chav) \ar[d]^(0.50){\approx}
\ar[r]^(0.50){}  &
H^{s(d-1)}(Z,\chav) \ar[d]^(0.50){\approx}
\ar[r]^(0.50){}  &
H^{s(d-1)}(Z) \ar[d]^(0.50){\approx}
\ar[r]^(0.50){}  &
H^{s(d-1)}(\chav) \ar[d]^(0.50){\approx} 
 \ar[r] & 0 \\
0 \ar[r] &
H_{r(d-1)+1}(Z,Z') \ar[r]^(0.50){}  &
H_{r(d-1)}(Z') \ar[r]^(0.50){\phi_{r(d-1)}}  &
H_{r(d-1)}(Z) \ar[r]^(0.50){}  &
H_{r(d-1)}(Z,Z') \ar[r] & 0
}}\end{array}
$$
with $r=n-1-s$.
The horizontal sequences are exact.
The (co)homology is with integral coefficients and the diagram
commutes up to sign \cite[Theorem~I.2.2]{Br}.

We deduce that $H_{r(d-1)}(Z,Z')\approx {\rm coker\,} \phi_{r(d-1)}$
which is isomorphic to $C_{r(d-1)}$ by \lemref{phi}. Therefore,
$H^{s(d-1)}(\cha{n}{d}(\ell))$ is free abelian with rank
$$
{\rm rank\,} H^{s(d-1)}(\cha{n}{d}(\ell))=
{\rm rank\,} C_{(n-1-s)(d-1)} = a_{s} \, .
$$
On the other hand, $H_{r(d-1)+1}(Z,Z')\approx \ker \phi_{r(d-1)}$
which, by \lemref{phi} is isomorphic (though not equal, in general) to 
$B_{r(d-1)}$. Therefore,
$H^{s(d-1)-1}(\cha{n}{d}(\ell))$ is free abelian with rank
$$
{\rm rank\,} H^{s(d-1)-1}(\cha{n}{d}(\ell))=
{\rm rank\,} B_{(n-1-s)(d-1)} = a_{n-s-2} \, .
$$
\end{proof}

\begin{Remark}\label{Rcohchawrong}\rm
\thref{T.cohcha} is wrong if $\ell$ is not dominated. For example, let 
$\ell=(1,1,1,\varepsilon)$ with $\varepsilon<1$. Then,
$\cha{4}{d}(\ell)$ is diffeomorphic to the unit tangent bundle $T^1S^{d-1}$
of $S^{d-1}$: a map $g\:\cha{4}{d}(\ell)\to T^1S^{d-1}$ is given by
$g(\rho)=(\rho(1),\hat\rho(2))$, where the latter is obtained from 
$(\rho(1),\rho(2))$ by the Gram-Schmidt orthonormalization process.
The map $g$ is clearly a diffeomorphism for $\varepsilon=0$ and 
the robot arm $F_{(1,1,1)}\:\bbs^3_d\to\bbr^d$ of \secref{S.robotarm} 
has no critical value in the disk $\{|x|<1\}\subset\bbr^d$. 
In particular, $\cha{4}{3}(\ell)$ is diffeomorphic to $SO(3)$,
and thus $H^2(\cha{4}{3}(\ell);\bbz)=\bbz_2$. What goes wrong is 
Point (b) of \lemref{phi}: for instance $A_2=0$, 
$B_2=H_2(Z,Z')\approx H^2(Z) = C_2\approx \bbz^3$ and, by the proof of 
\thref{T.cohcha}, $\phi\: H^2(Z')\to H^2(Z)$ must be injective with cokernel $\bbz_2$.
To obtain this fine result with our technique would require to control the
signs in \lemref{intnum}. 
\end{Remark}

\section{The manifold $V_d(\ell)$}\label{S.cVd}

Let $\ell\in\bbr_{>0}^n$ be a length-vector. In \cite{Ha,Ha2},
a manifold $V_d(\ell)$ is introduced, whose boundary is $\chav=\cha{n}{d}(\ell)$,
and Morse Theory on $V_d(\ell)$ provides some information on $\chav$.
In this section, we further study the manifold $V_d(\ell)$ in order to compute 
the ring $H^{(d-1)*}(\chav)$ when $d\geq 3$.

Presented as a submanifold of $Z=Z^n_d$, the manifold $V_d(\ell)$ is
$$
V_d(\ell)=\{\rho\in Z\mid
\sum_{i=1}^{n-1}\ell_i\rho(i)=t e_1 \hbox{ with } t\geq \ell_n\} \, .
$$
Observe that $V_d(\ell)$ is $O(d-1)$-invariant.
Let $g:V_d(\ell)\to\bbr$ defined by $g(z)=-|\sum_{i=1}^{n-1}\ell_iz_i|$.
The following proposition is proven in \cite[Th.~3.2]{Ha}.

\begin{Proposition}\label{Morse}
Suppose that the length-vector $\ell\in\bbr_{>0}^n$ is generic. Then
\renewcommand{\labelenumi}{(\roman{enumi})}
\begin{enumerate}
\item  $V_d(\ell)$ is a smooth $O(d-1)$-submanifold
of $Z$, of dimension $(n-2)(d-1)$, with boundary
$\chav$.
\item $g\:V_d(\ell)\to\bbr$
is an $O(d-1)$-equivariant Morse function, with critical points
$\{\rho_{\scr J} \mid J \hbox{ short and } n\in J\}$ 
(see \eqref{rhoj} for the definition of $\rho_{\scr J}$).
The index of $\rho_{\scr J}$ is equal to $(d-1)(|J|-1)$. 
\mancqfd
\end{enumerate}
\end{Proposition}

\begin{Corollary}\label{Morse-Cor}
The cohomology group $H^*(V_d(\ell);\bbz)$ is free abelian and
$$
{\rm rank\,} H^k(V_d(\ell);\bbz) =
\left\{
\begin{array}{lll}
a_{s} & \hbox{if } \kern 2mm k=s(d-1) \\
0 & \hbox{otherwise.}
\end{array}\right.
$$
\end{Corollary}

\begin{proof}
The number of critical point of $g$ is equal to $a_s$. \corref{Morse-Cor}
is then obvious if $d\geq 3$. When $d=2$, one uses \cite[Theorem~4]{FS},
the Morse function $g$ being $\tau$-invariant and its
critical set being the the fixed point set $V_d(\ell)^\tau$.
\end{proof}

For each $J\subset\unmi$, define the submanifold
$\calr_d(J)$ of $Z^n_d=Z$ by
$$
\calr_d(J)=\{\rho\in Z\mid
\rho(i)=e_1 \hbox{ if } i\notin J \} \approx (S^{d-1})^{J} \, .
$$
Consider the space
$$
\calr_d(\ell) = \bigcup_{J\in\shod} \calr_d(J) \subset Z \, .
$$
As $\shod$ is a simplicial complex, the family $\{[\calr_d(J)]\mid J\in\shod\}$
is a free basis for $H_*(\calr_d(\ell);\bbz)$ (homology classes of $\calr_d(J)$
in lower degrees coincide in $H_*(\calr_d(\ell))$ with $[\calr_d(J')]$ for $J'\subset J$).
Thus, $H_*(\calr_d(\ell))$ is free abelian and
\begin{equation}\label{rkHrd}
{\rm rank\,} H_k(\calr_d(\ell);\bbz) =
\left\{
\begin{array}{lll}
a_{s} & \hbox{if } \kern 2mm k=s(d-1) \\
0 & \hbox{otherwise.}
\end{array}\right.
\end{equation}

\begin{Lemma}\label{KinV}
For $d\geq 2$,  there exists a map
$\mu\:\calr_d(\ell)\to V_d(\ell)$ such that
$H^*\mu\:H^*(V_d(\ell);\bbz)\to H^*(\calr_d(\ell);\bbz)$ is a ring isomorphism.
\end{Lemma}

\begin{proof}
Let $J\in\shod$. Elementary Euclidean geometry shows that,
for $\rho\in\calr_d(J)$, there is a unique $\hat\rho\in V_d(\ell)$ satisfying
the three conditions
\renewcommand{\labelenumi}{(\alph{enumi})}
\begin{enumerate}
\item $\hat\rho(i)=\rho(i)$ if $i\in J$ and
\item   $|\hat\rho(\bar J)|=1$, where $\bar J = \unmi - J$.
\item $\llangle{\hat\rho(i)}{e_1}>0$ if $i\in\bar J$.
\end{enumerate}
This defines an embedding $\mu_{\scr J}\:\calr_d(J)\to V_d(\ell)$
by $\mu_{\scr J}(\rho)=\hat\rho$. An example is drawn below with $n=6$ and $J=\{1,2,3\}$
(the last segments $\ell_n\rho(n)=-\ell_ne_1$ of the configurations are not drawn).

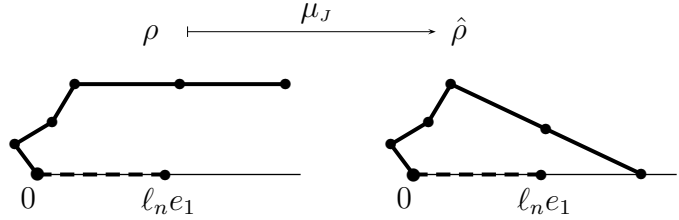
\sk{25}
\hskip 50mm
\begin{minipage}{6cm}
\setlength{\unitlength}{.05mm}
\begin{pspicture}(0,-2.1)(0,0)
\savebox{\boxJ}{
\put(0,0){\circle*{0.18}} 
\uput[-95](-0.1,0){$0$}
\psline[linewidth=1.5pt](0,0)(-0.3,0.4)(0.2,0.7)(0.5,1.2)
\put(-0.3,0.4){\circle*{0.14}} \put(0.2,0.7){\circle*{0.14}}
\put(0.5,1.2){\circle*{0.14}}
\psline[linewidth=0.5pt](0,0)(3.5,0)
\psline[linewidth=1.5pt,linestyle=dashed](0,0)(1.7,0)\put(1.7,0){\circle*{0.14}}
\uput[-85](1.7,0){$\ell_ne_1$}
}
\savebox{\boxbarJ}{
\psline[linewidth=1.5pt](0,0)(1.4,0)(2.8,0)
\put(1.4,0){\circle*{0.14}} \put(2.8,0){\circle*{0.14}}
} 
\put(-3,0){\usebox{\boxJ}}
\put(-2.5,1.2){\usebox{\boxbarJ}} 
\put(2,0){\usebox{\boxJ}}
\put(2.5,1.2){\rotatebox{-25.3}{\usebox{\boxbarJ}
 }}
\put(-1.6,1.8){$\rho$}\put(2.5,1.8){$\hat\rho$}
\psline[linewidth=0.3pt]{->}(-1,1.9)(2.3,1.9)
\psline[linewidth=0.3pt](-1,1.97)(-1,1.83)
\put(0.5,2.1){$\mu_{\scr J}$}
\end{pspicture}
\end{minipage}

We shall construct the map $\mu\:\calr_d(\ell)\to V_d(\ell)$ so that
its restriction to $\calr_d(J)$ is homotopic to $\mu_J$ for each $J\in\shod$.
Unfortunately, when $J\subset J'$, the restriction of $\mu_{J'}$ to $\calr_d(J)$
is not equal to  $\mu_J$ so the construction of $\mu$ requires some work.

For $J\in\shod$, consider the space of embeddings
$$
\algn_J=\big\{\alpha\:\calr_d(J)\to V_d(\ell)\mid \alpha(\rho) \hbox{ satisfies (a) and (c)}\big\}
$$
We claim that $\algn_J$ retracts by deformation onto its one-point subspace $\{\mu_J\}$. Indeed, let $\alpha\in\algn_J$ and let $\rho\in\calr_d(\ell)$. 
For $J\in\shod$, consider the space
$$
A_\rho = \{\zeta\:\bar J\to S^{d-1}\mid
\llangle{\zeta(i)}{e_1}>0 \hbox{ and }
\sum_{i\in J}\rho(i) +  \sum_{i\in\bar J}\zeta(i) = \lambda e_1 \hbox{ with } \lambda > 0  \} \, .
$$
Obviously, $\alpha(\rho)_{|\bar J} \in A_\rho$.
The space $A_\rho$ is a submanifold of $(S^{d-1})^{|\bar J|}$ which can be endowed with the
induced Riemannian metric. The parameter $\lambda$ provides a function 
$\lambda\:A_\rho \to \bbr$.
As usual, this is a Morse function with critical points the lined configurations
$\zeta(i)=\pm\zeta(j)$. But, as $\llangle{\zeta(i)}{e_1}>0$, the only critical point is
a maximum, the restriction of $\mu_J(\rho)$ to $\bar J$. Following the gradient line
at constant speed thus produces a deformation retraction of $A_\rho$ onto  $\mu_J(\rho)_{|\bar J}$.
The manifold $A_\rho$ and its gradient vector field depending smoothly on $\rho$,
this provides the required deformation retraction of $\algn_J$ onto $\{\mu_j\}$.

Let $\bcals_n$ be the first barycentric subdivision
of $\shod$. Recall that the vertices of $\bcals_n$
are the barycenters $\hat J \in |\shod|$ of the simplexes
$J$ of $\shod$, where $|\!\cdot\! |$ denotes the geometric realization.
A family $\{\hat J_0,\dots,\hat J_k\}$ of distinct barycenters forms a
$k$-simplex $\sigma\in \bcals_n$ if
$J_0\subset J_1\subset\cdots\subset J_k$. Set $\min\sigma= J_0$. For $J\in\shod$,
we also see $\hat J$ as a point of
$|\bcals_n|=|\shod|$.

Let us consider the quotient space:
\begin{equation}\label{defharrd}
\hat\calr_d(\ell) = \coprod_{\sigma \in \bcals_n} |\sigma|\times\calr_d(\min\sigma)
\Bigg/ \sim \ ,
\end{equation}
where $(t,\rho)\sim (t',\rho')$ if $\sigma<\sigma'$, $t=t'\in |\sigma|\subset  |\sigma'|$
and $\rho\mapsto\rho'$ under the inclusion
$\calr_d(\min\sigma)\hookrightarrow \calr_d(\min\sigma')$.
The projections onto the first factors in~\eqref{defharrd} provide a map
$q\:\hat\calr_d\to |\bcals_n|$ such that
$q^\mun(\hat J)=\{\hat J\}\times \calr_d(J)\approx \calr_d(J)$.
Over a $1$-simplex $e=\{\{J\},\{J,J'\}\}$ of $\bcals_n$,
one has
$q^\mun(\{J\})\approx\calr_d(J)$, $q^\mun(\{J'\})\approx\calr_d(J')$
and $q^\mun(|e|)$ is the mapping
cylinder of the inclusion  $\calr_d(J)\hookrightarrow \calr_d(J')$.

We now define a map $\hat\mu\:\hat\calr_d \to V_d(\ell)$
by giving its restriction
$$
\hat\mu^{k}\:q^\mun(|\bcals_d(\ell)^k|\to V_d(\ell)
$$
over the $k$-skeleton of $\bcals_n$. We proceed by induction on $k$.
The restriction of $\hat\mu$ to $q^\mun(\hat J)=\calr_d(J)$ is equal
to $\mu_J\in \algn_J$. This defines $\hat\mu^{0}$.
For an edge $e=\{\{J\},\{J,J'\}\}$, we use that
$\algn_{J}$ is contractible, as seen above.
The restriction of $\mu_{J'}$ to $\calr_d(J)$ is thus homotopic to
$\mu_J$ and we can use a homotopy to extend $\hat\mu^{0}$ over $|e|$.
Thus $\hat\mu^{1}$ is defined.
Suppose that $\hat\mu^{k-1}$ is defined. Let
$\sigma=\{\hat J_0,\dots,\hat J_k\}$ be a a $k$-simplex of $\bcals_d(\ell)$
with $\min\sigma=J_0$ and with boundary ${\rm Bd\,}\sigma$. As
$\algn_{J_0}$ is contractible, the restriction of $\hat\mu^{k-1}$ to
$q^\mun(|{\rm Bd\,}\sigma|)$ extends to $q^\mun(|\sigma|)$. This process permits us to define  $\hat\mu^{k}$.

Now the projections to the second factors in~\eqref{defharrd} give rise to a surjective map
$p\:\hat\calr_d(\ell)\to\calr_d(\ell)$. Let $x\in\calr_d(\ell)$. Let $J\in\shod$
minimal such that $x\in\calr_d(J)$. Then
$$
p^\mun(\{x\}) = |{\rm Star}(\hat J,\bcals_n)| \times  \{x\}
$$
is a contractible polyhedron. The preimages of points of $p$ are then all contractible and locally contractible,
which implies that $p$ is a homotopy equivalence \cite{Sm}. Using a homotopy
inverse for $p$ and the map $\hat\mu$, we get a map $\mu\:\calr_d(\ell)\to V_d(\ell)$.

For $J\in\shod$,
let us compose $\mu_J$ with the inclusion
$\beta\:V_d(\ell)\hookrightarrow Z$.
When $\rho\in\calr_d(J)$, the common value $\hat\rho(i)$ for $i\notin J$ is
not equal to $-(e_1,e_1,\dots,e_1)$. Using arcs of geodesics in $S^{d-1}$ enables us to construct
a homotopy from $\beta\pcirc\mu_J$ to the inclusion of $\calr_d(J)$ into
$Z$. This implies
that $H_*\mu\:H_*(\calr_d(\ell);\bbz)\to H_*(V_d(\ell);\bbz)$
is injective. By \corref{Morse-Cor} and \eqref{rkHrd},
$H_*\mu$ is an isomorphism. Hence,
$H^*\mu\: H^*(V_d(\ell);\bbz)\to H^*(\calr_d(\ell);\bbz)$ is a ring isomorphism.
\end{proof}

\begin{Remark}\rm
When $d\geq 3$, \lemref{KinV} implies that $\mu\:\calr_d(\ell)\to V_d(\ell)$
is a homotopy equivalence, since the spaces under consideration are simply connected.
We do not know if $\mu$ is also a homotopy equivalence when $d=2$.
\end{Remark}

\section{Proof of Theorm B}\label{SPthB}

Theorem B is a direct consequence of Propositions~\ref{LHCViso} and~\ref{togub}
below.

\begin{Proposition}\label{LHCViso}
Let $\ell\in\bbr_{>0}^n$ be a generic length-vector which is dominated.
Then the inclusion $\cha{n}{d}(\ell)\subset V_d(\ell)$
induces an injective ring homomorphism
\begin{equation}\label{LHCViso-eq}
H^*(\calr_d(\ell);\bbz)\approx
H^*(V_d(\ell);\bbz)\hookrightarrow
H^*(\cha{n}{d}(\ell);\bbz) \, .
\end{equation}
When $d\geq 3$ its image is equal to the subring $H^{(d-1)*}(\cha{n}{d}(\ell);\bbz)$.
\end{Proposition}

\begin{proof}
By \thref{T.cohcha} and its proof, the homomorphism
$H^{s(d-1)}(Z^n_d;\bbz)\to H^{s(d-1)}(\cha{n}{d}(\ell);\bbz)$
induced by the inclusion is surjective and
${\rm rank\,} H^{s(d-1)}(\cha{n}{d}(\ell);\bbz) = a_s$
(recall that $Z^n_d=Z$).
As the inclusion
$\cha{n}{d}(\ell)\subset Z^n_d$ factors through $V_d(\ell)$
the homomorphism $H^{s(d-1)}(V_d(\ell);\bbz)\to H^{s(d-1)}(\cha{n}{d}(\ell);\bbz)$
induced by the inclusion is also surjective.
As ${\rm rank\,} H^{s(d-1)}(V_{d}(\ell);\bbz) = a_s$ by \corref{Morse-Cor},
this proves the proposition.
\end{proof}


\begin{Remark}\rm
\proref{LHCViso} is wrong if $\ell$ is not dominated. For example, let
$\ell=(1,1,1,\varepsilon)$ with $\varepsilon<1$.
Then $a_1=3$, so $H^{d-1}(V_d(\ell);\bbz)\approx\bbz^3$. But, for $d=3$,
we saw in \remref{Rcohchawrong} that $H^2(\cha{4}{3}(\ell);\bbz)=\bbz_2$.
\end{Remark}


As in the introduction, consider the polynomial ring $\bbz_2[X_1,\dots,X_{n-1}]$ with
formal variables $X_1,\dots,X_{n-1}$.
If $J\subset\unmi$, we denote by $X_{J}\in \bbz_2[X_1,\dots,X_{n-1}]$
the monomial $\prod_{j\in J}X_j$.
Let $\cali(\shoti(\ell))$ be the ideal of $\bbz_2[X_1,\dots,X_{n-1}]$ generated
by the squares $X_i^2$ of the variables and the monomials $X_{J}$ for $J\notin\shoti(\ell)$ (non-simplex monomials).

\begin{Proposition}\label{togub}
The ring $H^*(\calr_d(\ell);\bbz_2)$ is isomorphic to the quotient ring\\
$\bbz_2[X_1,\dots,X_{n-1}]\big/\cali(\shoti(\ell))$
(The degree of $X_i$ being $d-1$).
\end{Proposition}

\begin{proof}
The coefficients of the (co)homology groups are $\bbz_2$ and
are omitted in the notation.
Consider the inclusion $\beta\:V_d(\ell)\hookrightarrow Z=Z^n_d$.
The map $\rho\mapsto (\rho(1),\dots,\rho(n-1))$ is a diffeomorphism
from $Z$ to $(S^{d-1})^{n-1}$. Using this identification,
the homology $H_*(Z)$ is the $\bbz_2$-vector space with basis
the classes
$[\calr_d(I))]$ for $I\subset\unmi$.
(To compare with the basis of \lemref{zjbase}, the submanifolds $R_d(J)$ and $W_{\bar J}$
are isotopic, where $\bar J$ is the complement of $J$ in $\unm$.)
The homology $H_*(\calr_d(\ell))$ has basis
$[\calr_d(J)]$ for $J\in\shoti(\ell)$. The homomorphism
$H_*\beta\:H_*(\calr_d(\ell))\to H_*(Z)$ is induced
by the inclusion of the above bases. Hence, $H_j\beta\:H_j(\calr_d(\ell))\to H_j(Z)$
is injective and ${\rm coker\,}H_j$ is freely generated by the classes 
$[\calr_d(J)]$ for $|J|=j$ and $J\notin\shoti(\ell)$.

In particular, the classes $[\calr_d(\{i\})]$, for $i=1,\dots,n-1$,
form a basis of $H_{d-1}(Z)$.
Let $\{\xi_1,\dots,\xi_{n-1}\}\in H^{d-1}(Z)=
{\rm hom}(H_{d-1}(Z),\bbz_2)$ be the Kronecker dual basis.
By the K\"unneth formula,
the correspondence $X_i\mapsto\xi_i$ extends to a ring isomorphism
$\bbz_2[X_1,\dots,X_{n-1}]\stackrel{\approx}{\to} H^*(Z)$. 
The the family of monomials $\{X_J\mid J\subset\unmi\}$ is sent 
to the the Kronecker dual basis to $\{[\calr_d(J)]\mid J\subset\unmi\}$.
The properties
of $H_*\beta$ mentioned above then imply that the composed ring homomorphism
$$
\bbz_2[X_1,\dots,X_{n-1}]\stackrel{\approx}{\to} H^*(Z)
\stackrel{H^*\beta}{\longrightarrow}  H^*(\calr_d(\ell))
$$
is surjective with kernel $\cali(\shoti(\ell))$.
\end{proof}

The proof of Theorem~B is thus complete which, as seen in the introduction, 
implies Theorem~A.

\section{Comments}\label{S.comment}

\begin{ccote}\label{C.nota}\rm
The authors are trying to unify the notations used for the various posets of short subsets.
Our notation $\shoti\subset\shod\subset\cals$ are identical to that of \cite{Sz}. In \cite{HK},
$\shod$ is denoted by $\cals_n$ but, in the more recent papers \cite{HR, Ha2},
$\cals_n=\{J\in \cals\mid n\in J\}$. This is not used here but could have been naturally in
e.g. \thref{T.cohcha}.
\end{ccote}

\begin{ccote}\rm
When $d=2$, Assertion \eqref{LHCViso-eq} still holds true but not the last assertion of 
\proref{LHCViso}. The image $\calj^n_2(\ell)$ of the homomorphism 
$H^*(V_2(\ell);\bbz)\to H^*(\cha{n}{2}(\ell);\bbz)$ induced by the inclusion is just some
subring of $H^*_{(1)}(\cha{n}{2}(\ell);\bbz)$, where the latter denotes the
subring of $H^*(\cha{n}{2}(\ell);\bbz)$ generated by the elements of degree $1$. 
For length-vectors such that $\calj^n_2(\ell)=H^*_{(1)}(\cha{n}{2}(\ell);\bbz)$,
our proof of Theorem~B (and then of Theorem~A) holds. Such length-vectors are called
normal in~\cite{FHS}. 
\end{ccote}

\begin{ccote}\rm
The ring structure of $H^*(\cha{n}{d}(\ell);\bbz_2)$ is necessary to differentiate the chain spaces
up to diffeomorphism:
the Betti numbers are not enough. The first example occurs for $n=6$ with
$\ell=(1,1,1,2,3,3)$ and $\ell'=(\varepsilon,1,1,1,2,2)$, where $0<\varepsilon<1$.
(The chamber of $\ell$ is $\langle 632,64 \rangle$ and that of $\ell'$ is $\langle 641 \rangle$
, see \cite[Table~C]{Ha2}.)
Then, $\shoti(\ell)$ and $\shoti(\ell')$ are both graphs with $4$ vertices and $3$ edges.
Therefore, $a_s(\ell)=a_s(\ell')$ for all $s$ which, by \thref{T.cohcha}, implies that
$\cha{6}{d}(\ell)$ and $\cha{6}{d}(\ell')$ have the same Betti numbers. However,
$\shoti(\ell)$ and $\shoti(\ell')$ are not poset isomorphic:
the former is not connected while the latter is. 
\end{ccote}

\begin{ccote}\rm
It would be interesting to know if, in Theorem A, the ring $\bbz_2$ could be replaced by any
other coefficient ring. 
In the corresponding result for spatial polygon spaces $\nua{n}{3}(\ell)$, which are distinguished
by their $\bbz_2$-cohomology rings if $n>4$ \cite[Theorem~3]{FHS}, 
the ring $\bbz_2$ cannot be replaced by $\bbr$.
Indeed, $\nua{5}{3}(\varepsilon,1,1,1,2)\approx \bbc P^2\sharp\bar \bbc P^2$ while
$\nua{5}{3}(\varepsilon,\varepsilon,1,1,1)\approx S^2\times S^2$ 
($\varepsilon$ small; see \cite[Table~B]{Ha2}). 
These two manifolds have non-isomorphic $\bbz_2$-cohomology rings but isomorphic
real cohomology rings. 
One can of course replace $\bbz_2$ by $\bbz$ in Theorem~A
since, by \thref{T.cohcha}, $H^*(\cha{n}{d}(\ell);\bbz)$ determines $H^*(\cha{n}{d}(\ell);\bbz_2)$ 
when $\ell$ is dominated.
\end{ccote}

\begin{ccote}\rm
We do not know if Theorem A is true for generic length vectors which are not dominated.
The techniques developped in \cite{FF} might useful to study this more general case.
\end{ccote}

\begin{ccote}\rm
Let $K$ be a flag simplicial complex (i.e. if $K$ contains a graph $L$ isomorphic
to the $1$-skeleton of a $r$-simplex, then $L$ is
contained in a $r$-simplex of $K$).
Then the complex $\calr_1(K)$ is the Salvetti complex of the right-angled
Coxeter group determined by the $1$-skeleton of $K$, see \cite{Cha}.
\end{ccote}

\sk{3}\noindent{\bf Acknowledgements: }\rm
The second author would like to thank E. Dror-Farjoun for useful discussions.


\end{document}